\documentclass{article}
\usepackage[T2A]{fontenc}
\usepackage[utf8]{inputenc}
\usepackage[russian]{babel}

\usepackage{graphicx}
\usepackage{amsmath}
\usepackage{amsfonts}
\usepackage{wrapfig}
\usepackage{float}
\usepackage{bbm}
\usepackage{xcolor}
\usepackage{nicefrac}
\usepackage{mathtools}
\usepackage{hyperref}

\makeatletter 
\newcommand\mynobreakpar{\par\nobreak\@afterheading} 
\makeatother

\DeclarePairedDelimiter{\ceil}{\lceil}{\rceil}

\usepackage[ruled,vlined]{algorithm2e}

\usepackage{amsthm}
\usepackage{amsmath}
\usepackage{amssymb}

\newtheorem{theorem}{Теорема}[section]
\newtheorem{corollary}{Следствие}[theorem]

\newtheorem{proposition}[theorem]{Утверждение}

\newsavebox{\mybox}
\newlength{\mywidth}
\newlength{\myheight}
\newlength{\myline}
\newlength{\myoffset}
\newcommand{\mysqrt}[1]%
{\setlength{\myline}{.2ex}%
\addtolength{\myline}{.06pt}%
\setlength{\myoffset}{.9em}
\addtolength{\myoffset}{-2pt}
\savebox{\mybox}{$\displaystyle\sqrt{#1}$}%
\settoheight{\myheight}{\usebox{\mybox}}%
\addtolength{\myheight}{-.3ex}
\settowidth{\mywidth}{\usebox{\mybox}}%
\addtolength{\mywidth}{-\myoffset}%
 \rlap{\usebox{\mybox}}\hspace{\myoffset}{\raisebox{\myheight}{\rule{\mywidth}{\myline}}}}

\usepackage[left=3cm,right=3cm,
    top=2cm,bottom=2cm,bindingoffset=0cm]{geometry}

\title{Ускоренные проксимальные оболочки: применение к покомпонентому методу\footnote{Исследования Д.А. Пасечнюка были поддержаны стипендией А.М. Райгородского в области оптимизации и грантом РФФИ 19-31-51001 Научное наставничество. Исследования А.С. Аникина были поддержаны грантом РФФИ 18-29-03071 мк. Работа В.В. Матюхина выполнена при поддержке Министерства науки и высшего образования Российской Федерации (госзадание) №075-00337-20-03, номер проекта 0714-2020-0005.}}
\author{Пасечнюк Д.А., Аникин А.С., Матюхин В.В.}

\begin{document}

\maketitle

\begin{center}
\begin{minipage}{0.75\textwidth}
    \quadСтатья посвящена одному частному случаю применения универсальных ускоренных проксимальных оболочек для получения вычислительно эффективных ускоренных вариантов методов, использующихся для решения различных частных постановок оптимизационных задач. В данной работе предлагается проксимально ускоренный покомпонентный градиентный метод с эффективной алгоритмической сложностью итерации, позволяющий существенно учитывать разреженность решаемой задачи, и рассматривается пример применения предлагаемого подхода для решения задачи оптимизации функции вида SoftMax, для которой описываемый метод позволяет ослабить зависимость вычислительной сложности решения от размерности $n$ задачи в $\mathcal{O}(\sqrt{n})$ раз, и демонстрирует на практике более быструю по сравнению со стандартными методами сходимость. 
    \newline

    \textbf{Ключевые слова:} проксимальный ускоренный метод, каталист, ускоренный покомпонентный метод, SoftMax, LogSumExp. \newline
    
\end{minipage}
\end{center}

\section{Введение}

Одним из важнейших теоретических результатов в выпуклой оптимизации является разработка ускоренных методов оптимизации \cite{nesterov2018lectures}. На начальном этапе развития этой концепции было предложено множество ускоренных вариантов различных методов, применяющихся к решению многих задач выпуклой оптимизации, однако каждый такой случай требовал отдельного, частного рассмотрения возможности ускорения, ввиду чего предлагаемые конструкции были существенно различны и не позволяли предполагать способ их обобщения. Важным шагом к разработке универсальной схемы ускорения методов оптимизации стали работы, в которых предлагался и исследовался метод, названный каталист, основанный на идее ускоренного проксимального градиентного метода \cite{parikh2014proximal, rockafellar1976monotone} и позволяющий ускорять другие методы оптимизации, используя их для последовательного решения ряда резуляризованных по Моро-Иосиде вспомогательных задач \cite{lin2015universal, lin2017catalyst}. В продолжение этих идей в дальнейшем было предложено множество вариантов применения данного метода и его модификаций \cite{ivanova2019adaptive, kulunchakov2019generic, paquette2017catalyst}. Среди наиболее свежих, на момент написания данной статьи, результатов были также описаны обобщения обсуждаемого подхода на тензорные методы \cite{bubeck2019near, doikov2020contracting, monteiro2013accelerated, gasnikov2020accelerated}. Соответствующее представление ускоренной проксимальной оболочки, если опираться на известные авторам сведения, является наиболее общим из описанных в литературе, и потому в данной работе внимание будет обращено прежде всего именно на методы, предложенные в последней из цитируемых выше работ. 

Основной мотив данной работы состоит в том, чтобы описать возможности практического применения универсальных ускоренных проксимальных оболочек для конструирования вычислительно и оракульно эффективных методов оптимизации. Рассмотрим классический покомпонентный метод \cite{bubeck2014convex}, итерация которого для выпуклой функции $f: \mathbb{R}^n \rightarrow \mathbb{R}$ имеет вид:
$$x_{k+1}^i = x_k^i - \eta \nabla_i f(x_k), \quad i \sim \mathcal{U}\{1,...,n\},\;\;\eta > 0.$$
Одним из многих приложений данного метода является оптимизация функционалов, вычисление одной компоненты градиента которых значимо эффективнее, чем вычисление полного вектора градиента; в частности, многие задачи в случае разреженных постановок удовлетворяют данному условию. Однако оракульная сложность данного метода при условии остановки метода при достижении $\varepsilon$-малости невязки по значению функции составляет $\mathcal{O}\left(n \frac{\overline{L} R^2}{\varepsilon}\right)$, где $R^2 = \|x_0 - x_*\|_2^2$, $\overline{L}=\frac{1}{n} \sum_{i=1}^n L_i$~--- среднее констант Липшица компонент градиента, притом эта оценка не является оптимальной для класса выпуклых задач. Рассмотрим теперь ускоренный покомпонентный метод, предложенный Ю.Е.~Нестеровым \cite{nesterov2017efficiency},~--- оракульная сложность данного метода соответствует оптимальной оценке: $\mathcal{O}\left(n\sqrt{\frac{\widetilde{L} R^2}{\varepsilon}}\right)$, где $\sqrt{\widetilde{L}} = \frac{1}{n} \sum_{i=1}^n \sqrt{L_i}$ ~--- среднее квадратных корней из констант Липшица компонент градиента. Вместе с тем, ситуация кардинально меняется в случае рассмотрения алгоритмической сложности метода: пусть даже вычисление одной компоненты градиента имеет сложность $\mathcal{O}(s)$, $s \ll n$, сложность итерации ускоренного покомпонентного метода будет составлять $\mathcal{O}(n)$, в отличие от стандартного метода, сложность итерации которого есть $\mathcal{O}(s)$,~--- содержательно это означает, что степень разреженности задачи при применении ускоренного покомпонентного метода не влияет существенно на сложность алгоритма, и кроме того сложность в таком случае квадратично зависит от размерности задачи: вместе это в некоторой степени обесценивает применение покомпонентного метода в данном случае. Таким образом, интересной задачей является построение ускоренного покомпонентного метода, сложность итерации которого, как и в стандартном варианте метода, составляет $\mathcal{O}(s)$, при сохранении оптимальной оракульной сложности~--- в данной работе это удаётся осуществить благодаря применению универсальной ускоренной проксимальной оболочки ``ускоренный метаалгоритм`` \cite{gasnikov2020accelerated}.

Данная статья состоит из двух основных разделов. В разделе \ref{section1} описываются теоретические результаты о сходимости и алгоритмической сложности покомпонентного метода, ускоренного путём применения оболочки ``ускоренный метаалгоритм``. В разделе \ref{section2} на примере задачи оптимизации функционала вида SoftMax экспериментально проверяется эффективность метода в отношении времени его работы, описываются возможности его вычислительно эффективной имплементации и осуществляется сравнение со стандартными методами. 

\section{Теоретические гарантии} \label{section1}

Рассмотрим следующую задачу оптимизации функции $f: \mathbb{R}^n \rightarrow \mathbb{R}$:
$$\min_{x \in \mathbb{R}^n} f(x),$$
при таких условиях:
\begin{enumerate}
    \item $f$ дифференцируема на $\mathbb{R}^n$;
    \item $f$ выпукла на $\mathbb{R}^n$;
    \item $\nabla_i f$ удовлетворяет условию Липшица: $\exists L_{i} \in \mathbb{R}:\;\forall x \in \mathbb{R}^n, u \in \mathbb{R}$
    $$\left|\nabla_{i} f\left(x+u e_{i}\right)-\nabla_{i} f(x)\right| \leq  L_i|u|,$$
    где $e_i$~--- $i$-тый орт базиса, $i \in \{1,...,n\}$;
    \item $\nabla f$ удовлетворяет условию Липшица с константой $L$.
\end{enumerate}

Обратимся к содержанию работы \cite{gasnikov2020accelerated}, где предложен общий вариант ``ускоренного метаалгоритма`` решения задач выпуклой оптимизации для композитных функционалов вида $F(x) = f(x) + g(x)$. Для рассматриваемой постановки задачи такая общность не требуется, достаточно применить частный случай описанной схемы при $p=1$, $f \equiv 0$ (используются обозначения соответствующей работы), в котором описанная оболочка принимает вид ускоренного проксимального метода. Псевдокод используемого метода представлен в листинге \ref{am}.

\begin{wrapfigure}{r}{0.51\textwidth}
\begin{minipage}{0.51\textwidth}
\begin{algorithm}[H] \label{am}
\SetAlgoLined
    \textbf{Вход:} $H > 0$, $x_0 \in \mathbb{R}^n$\;
    \vspace{0.2cm}

    $\lambda \leftarrow \nicefrac{1}{2H}$\;
    $A_0 \leftarrow 0$; $v_0 \leftarrow x_0$\;
    \vspace{0.2cm}
    
    \For{k = 0, ..., $\widetilde{N}-1$}
    {
        $\displaystyle a_{k+1} \leftarrow \frac{\lambda + \sqrt{\lambda^2 + 4 \lambda A_k}}{2}$\;
        $A_{k+1} \leftarrow A_k + a_{k+1}$\;
        
        \vspace{0.2cm}
        $\displaystyle \widetilde{x}_k \leftarrow \frac{A_k v_k + a_{k+1} x_k}{A_{k+1}}$\;
        
        \vspace{0.2cm}
        Посредством запуска метода $\mathcal{M}$\\
        найти с точностью $\varepsilon$ по аргументу\\
        решение вспомогательной задачи:\\
        \vspace{0.1cm}
        $\displaystyle v_{k+1} \in \text{Arg}^{\varepsilon} \min_{y \in \mathbb{R}^n} \left\{ f(y) + \frac{H}{2}\|y - \widetilde{x}_k\|^2_2 \right\}$\;
        \vspace{0.5cm}
        
        $\displaystyle x_{k+1} \leftarrow x_k - a_{k+1} \nabla f(v_{k+1})$\;
    }
    
    \Return $v_{\widetilde{N}}$\;
    
    \caption{Ускоренный метаалгоритм для метода $\mathcal{M}$ первого порядка}
\end{algorithm}
\vspace{0.5cm}

\begin{algorithm}[H] \label{cdm}
\SetAlgoLined
    \textbf{Вход:} $y_0 \in \mathbb{R}^n$\;
    \vspace{0.2cm}
    
    $Z \leftarrow \sum_{i=1}^n (H + L_i)$\;
    $p_i \leftarrow (H + L_i) / Z,\quad i \in \{1,...,n\}$\;
    Дискретное вероятностное\\ распределение $\pi$
    с вероятностями $p_i$\;
    
    \vspace{0.2cm}
    
    \For{k = 0, ..., $N-1$}
    {
        $i \sim \pi\{1,...,n\}$\;
        $y_{k+1} \leftarrow y_k$\;
        $\displaystyle y_{k+1}^i = y_k^i - \frac{1}{H + L_i} \nabla_i F(y_k)$\; 
    }
    
    \Return $y_N$\;
    
    \caption{Покомпонентный метод}
\end{algorithm}
\end{minipage}
\vspace{0.5cm}
\end{wrapfigure}

Прежде чем сформулировать какие-либо результаты о сходимости, следует подробнее рассмотреть вопрос о решении вспомогательной задачи~--- её аналитическое решение доступно лишь в редких случаях, и потому необходимо решать её численными методами, а значит неточно. Вспомогательную задачу допустимо решать до выполнения следующего условия останова (\cite{kamzolov2020optimal}, Appendix B):
\begin{equation} \label{monteiro}
    \left\|\nabla\left\{F(y_\star) := f(y_\star) + \frac{H}{2}\|y_\star-\widetilde{x}_k\|^2_2\right\}\right\|_2 \leq
\end{equation}
$$\leq \frac{H}{2}\|y_\star-\widetilde{x}_k\|_2.$$
Ввиду того, что $\|\nabla F(y_*)\|_2 = 0$, а также ввиду $(L + H)$-липшицевости $\nabla F$, имеем:
\begin{equation} \label{smooth}
    \|\nabla F(y_\star)\|_2 \leq (L + H) \|y_\star - y_*\|_2.
\end{equation}
Выписав неравенство треугольника: 
$\|\widetilde{x}_k - y_*\|_2 - \|y_\star - y_*\|_2 \leq \|y_\star - \widetilde{x}_k\|_2$,
и воспользовавшись вместе неравенствами \eqref{monteiro}, \eqref{smooth}, получаем окончательное условие останова:
\begin{equation} \label{crit}
    \|y_\star - y_*\|_2 \leq \frac{H}{3H + 2L} \|\widetilde{x}_k - y_*\|_2.
\end{equation}
Содержательно отсюда следует, что необходимая точность решения вспомогательной задачи по аргументу не зависит от требуемой точности решения общей задачи, что позволяет значимо упростить получение дальнейших результатов.

Рассмотрим теперь основной применяемый для решения вспомогательных задач метод: содержание покомпонентного метода \cite{nesterov2012efficiency} (в частном случае $\gamma = 1$) представлено в листинге \ref{cdm}. Для данного метода в случае рассматриваемых вспомогательных задач справедлив результат:

\begin{theorem}{(\cite{bubeck2014convex}, theorem 6.8)} \label{cdm_conv}
    Пусть $F$ является $H$-сильно выпуклой относительно $\|\cdot\|_2$. Тогда для последовательности $\{y_k\}_{k=1}^N$, генерируемой покомпонентным методом, выполняется
    \begin{equation} \label{f_conv}
        \mathbb{E}[F(y_N)] - F(y_*) \leq \left(1 - \frac{1}{\kappa}\right)^N (F(y_0) - F(y_*)),\quad \text{где}\quad\kappa = \frac{H}{Z},\;\;Z = \sum_{i=1}^n (H + L_i).
    \end{equation}
\end{theorem}

\noindentИспользуя данный результат, сформулируем утверждение о числе итераций покомпонентного метода, достаточном для выполнения полученного выше условия останова.

\begin{corollary} \label{cor_num}
    Матожидание $\mathbb{E}[y_N]$ точки, являющейся результатом работы покомпонентного метода, удовлетворяет условию \eqref{crit} достижения достаточной точности решения вспомогательной задачи ускоренного метаалгоритма в том случае, если для числа итераций метода выполнено
    \begin{equation}\label{cor}
        N \geq N(\widetilde{\varepsilon}) = \ceil[\Bigg]{\frac{Z}{H} \ln{\left\{\left(1 + \frac{L}{H}\right) \left(3 + \frac{2L}{H}\right)^2\right\}} }, \quad \text{где}\quad\widetilde{\varepsilon} = \frac{H}{2} \left(\frac{H}{3H + 2L}\right)^2 \|y_0 - y_*\|_2^2.
    \end{equation}
\end{corollary}
\begin{proof}
    
    Ввиду $(H+L)$-липшицевой гладкости функции $F$ верно: $\displaystyle F(y_0) - F(y_*) \leq \frac{H+L}{2} \|y_0 - y_*\|_2^2$. Используя это неравенство вместе с оценкой \eqref{f_conv}, можем выписать условие достижения заданной точности $\widetilde{\varepsilon}$ по функции: $\displaystyle \frac{H+L}{2} \left(1 - \frac{1}{\kappa}\right)^{N} \|y_0 - y_*\|_2^2 \leq \widetilde{\varepsilon}$. Также верно $1 - 1/\kappa \leq \exp\{-1/\kappa\}$, и, значит: 
    $$\displaystyle \frac{H+L}{2} \exp\{\kappa / N\} \|y_0 - y_*\|_2^2 \leq \widetilde{\varepsilon}.$$ 
    Логарифмируя и подставляя выражение для $\kappa$, получаем выражение для числа итераций от $\widetilde{\varepsilon}$:
    
    \begin{equation} \label{itt}
        N(\widetilde{\varepsilon}) = \ceil[\Bigg]{\frac{Z}{H} \ln {\left\{\frac{(H+L) \|y_0 - y_*\|_2^2}{2 \widetilde{\varepsilon}}\right\}}}.
    \end{equation}
    Ввиду $H$-сильной выпуклости функции $F$ верно $\displaystyle \|\overline{y}_N - y_*\|_2^2 \leq \frac{2}{H} (F(\overline{y}_N) - F(y_*))$, где $\overline{y}_N = \mathbb{E}[y_N]$. Функция $F$ выпукла, следовательно, по неравенству Йенсена, $F(\overline{y}_N) \leq \mathbb{E}[F(y_N)]$, и отсюда, вместе с \eqref{crit} получаем достаточное условие достижения решения вспомогательной задачи:
    $$\mathbb{E}[F(y_N)] - F(y_*) \leq \frac{H}{2} \left(\frac{H}{3H + 2L}\right)^2 \|y_0 - y_*\|_2^2.$$
    Подставляя в формулу \eqref{itt} вместо $\widetilde{\varepsilon}$ выражение из правой части данного неравенства, непосредственно приходим к выражению из утверждения.
\end{proof}

Теперь, когда полностью прояснён вопрос о требуемой точности и оракульной сложности решения вспомогательной задачи с помощью покомпонентного метода, можно перейти к результатам о сходимости ускоренного метаалгоритма. Для используемого условия останова \eqref{crit} метода, решающего вспомогательную задачу, справедлив следующий результат о сходимости ускоренного метаалгоритма:

\begin{theorem}{(\cite{gasnikov2020accelerated}, теорема 1)} \label{am_conv}
    При $H > 0$, для последовательности $\{v_k\}_{k=1}^{\widetilde{N}}$, генерируемой ускоренным метаалгоритмом, использующим для решения вспомогательной задачи некоторый не стохастический метод, выполняется
    \begin{equation} \label{am_th}
        f(v_{\widetilde{N}}) - f(x_*) \leq \frac{48}{5} \frac{H \|x_0 - x_*\|_2^2}{\widetilde{N}^2}.
    \end{equation}
\end{theorem}

\noindentНа основании последнего утверждения можно сформулировать теорему о сходимости ускоренного метаалгоритма в случае применения стохастического метода и, в частности, покомпонентного градиентного спуска.

\begin{theorem} \label{am_stoch_conv}
    При $H > 0$, для некоторого $0 < \delta < 1$, точка $v_{\widetilde{N}}$, являющаяся результатом работы ускоренного метаалгоритма, использующего для решения вспомогательной задачи покомпонентный метод, решающий вспомогательную задачу $N_{\delta}$ итераций, удовлетворяет условию
    \begin{equation*}
        Pr(f(v_{\widetilde{N}}) - f(x_*) < \varepsilon) \geq 1 - \delta
    \end{equation*}
    в случае, если 
    \begin{equation} \label{iters_out_inn}
        \widetilde{N} \geq \ceil[\Bigg]{\frac{4 \sqrt{15}}{5} \sqrt{\frac{H \|x_0 - x_*\|_2^2}{\varepsilon}}}, \quad N_{\delta} \geq N\left(\frac{\widetilde{\varepsilon}\delta}{\widetilde{N}}\right) = \ceil[\Bigg]{\frac{Z}{H} \ln { \left\{\frac{\widetilde{N}}{\delta} \left(1 + \frac{L}{H}\right) \left(3 + \frac{2L}{H}\right)^2\right\}} }.
    \end{equation}
\end{theorem}

\begin{proof}
    В следствии \ref{cor_num} представлена оценка числа итераций, достаточного для выполнения следующего условия на матожидание значения функционала в результирующей точке метода:
    $$\mathbb{E}[F(y_{N(\widetilde{\varepsilon})})] - F(y_*) \leq \widetilde{\varepsilon}.$$
    Воспользуемся неравенством Маркова и получим формулировку данного условия в терминах оценки вероятности больших отклонений \cite{anikin2015modern}: заведомо выберем допустимое значение вероятности невыполнения поставленного условия, так чтобы $0 < \delta/\widetilde{N} < 1$, где $\widetilde{N}$ выражается из \eqref{am_th}; тогда
    $$Pr\left(F\left(y_{N(\widetilde{\varepsilon} \delta/\widetilde{N})}\right) - F(y_*) \geq \widetilde{\varepsilon}\right) \leq \frac{\delta}{\widetilde{N}} \cdot \frac{\mathbb{E}\left[F\left(y_{N(\widetilde{\varepsilon} \delta/\widetilde{N})}\right)\right] - F(y_*)}{\widetilde{\varepsilon} \cdot \delta/\widetilde{N}} = \frac{\delta}{\widetilde{N}}.$$
    Поскольку вероятность того, что полученное решение некоторой отдельно взятой вспомогательной задачи не будет удовлетворять поставленному условию, равна $\delta / \widetilde{N}$, значит вероятность того, что за $\widetilde{N}$ итераций ускоренного метаалгоритма условие будет невыполнено хотя бы для одной из задач, есть $\widetilde{N} \cdot \delta / \widetilde{N} = \delta$, откуда и следует доказываемое утверждение.
\end{proof}

Далее, объединяя оценки, приводимые в теореме \ref{am_stoch_conv}, можем получить асимптотическую оценку на общее число итераций покомпонентного метода, достаточное для решения рассматриваемой оптимизационной задачи с некоторой заданной точностью, а также оценку оптимального параметра $H$:

\begin{corollary}
    Для того чтобы точка $v_{\widetilde{N}}$, являющаяся результатом работы ускоренного метаалгоритма, удовлетворяла условию 
    $$Pr(f(v_{\widetilde{N}}) - f(x_*) < \varepsilon) \geq 1 - \delta,$$
    достаточно выполнить в сумме
    \begin{equation} \label{total_iters}
        \hat{N} \geq \widetilde{N} \cdot N_\delta = \mathcal{O}\left(\frac{Z \|x_0 - x_*\|_2}{\sqrt{H}} \cdot \frac{1}{\varepsilon^{1/2}} \log\left\{\frac{1}{\varepsilon^{1/2} \delta}\right\}\right)
    \end{equation}
    итераций покомпонентного метода для решения вспомогательной задачи. При этом, оптимально значение параметра $H$ регуляризации вспомогательной задачи следует выбирать как $H \simeq \frac{1}{n} \sum_{i=1}^n L_i$ ($\simeq$ обозначает равенство с точностью до малого множителя порядка $\log$).
\end{corollary}

\begin{proof}
    Выражение для $\hat{N}$ можно получить путём прямой подстановки одной из оценок, приводимых в \eqref{iters_out_inn}, в другую, и их последующего умножения. Если исключить из рассмотрения малый множитель порядка $\log(L / H)$, константа в оценке будет зависеть от $H$ как:
    $$\sqrt{H} \cdot \frac{Z/n}{H} = \sqrt{H} \left(1 + \frac{\frac{1}{n} \sum_{i=1}^n L_i}{H}\right).$$
    Минимизируя представленное выражение по $H$, получаем указанный результат.
\end{proof}

Рассмотрим теперь подробнее вопрос об алгоритмической сложности предложенного ускоренного метода покомпонентного градиентного спуска. Очевидным является следующее утверждение:

\begin{proposition} \label{constr}
    Алгоритмическая сложность рассматриваемого метода составляет
    $$T = \mathcal{O}\left(\widetilde{N} (T_{out} + N_{\delta} T_{inn})\right),$$
    где $T_{out}$~--- амортизированная оценка сложности вычислений, производимых на итерации ускоренного метаалгоритма, $T_{inn}$~--- амортизированная оценка сложности итерации покомпонентного метода.
\end{proposition}

\noindentОттолкнувшись от него, сформулируем результат о вычислительной сложности метода:

\begin{theorem}
    Пусть сложность вычисления одной компоненты градиентна $f$ составляет $\mathcal{O}(s)$. Тогда алгоритмическая сложность рассматриваемого метода есть
    $$T = \mathcal{O}\left(sn \cdot \sqrt{\frac{\overline{L} \|x_0 - x_*\|_2^2}{\varepsilon}} \log\left\{\frac{1}{\varepsilon^{1/2} \delta}\right\}\right),\quad\text{где}\quad \overline{L} = Z / n = \frac{1}{n} \sum_{i=1}^n L_i.$$
\end{theorem}

\begin{proof}
    Переформулируем оценку из утверждения \ref{constr} следующим образом:
    $$T = \mathcal{O}\left(\hat{N} \cdot T_{iter}\right),$$
    где $T_{iter}$~--- амортизированная оценка сложности элементарной итерации ускоренного метаалгоритма, то есть итерации, которая может быть и внутренней итерацией покомпонентного метода, и основной итерацией метаалгоритма. Сложность основной итерации метаалгоритма определяется в первую очередь вычислением полного градиента $f$, а сложность этой процедуры (из условия теоремы) есть $\mathcal{O}(sn)$. В то же время, ввиду $Z = n \overline{L}$, верно также $N_\delta = \widetilde{\mathcal{O}}(n)$, где символ $\widetilde{\mathcal{O}}(\cdot)$ означает то же, что и $\mathcal{O}(\cdot)$, но с возможным присутствием множителей порядка $\log(\cdot)$. Поскольку основная итерация метаалгоритма исполняется через каждые $N_\delta$ элементарных итераций, где $N_\delta$ постоянно, из этого с помощью любого из методов амортизационного анализа тривиально получается амортизированная оценка сложности основной итерации метаалгоритма, составляющая $\mathcal{O}(s)$. Сложность итерации покомпонентного метода (если вместо копирования значений точки выполнять операции ``на месте``, что для данной конструкции вполне допустимо) определяется вычислением одной компоненты градиента, и составляет также $\mathcal{O}(s)$. Отсюда получаем $T_{iter} = \mathcal{O}(s)$. Используя оценку \eqref{total_iters} и подставляя оптимальное значение $H$, получаем приведённую сложность метода. 
    
    Заметим также, что сложность метода по памяти при этом составляет $\mathcal{O}(n)$, также как и сложность предварительных вычислений (для покомпонентного метода нет необходимости выполнять их каждый раз заново). 
\end{proof}

Сравним полученные для предложенного метода оценки с оценками других методов, которые могут быть использованы для решения задач в описываемой постановке: быстрого градиентного метода (FGM), классического покомпонентного спуска (CDM), ускоренного покомпонентного спуска в варианте Ю.Е.~Нестерова (ACDM) и предложенного в данной работе подхода (Catalyst CDM). Оценки приведены в таблице \ref{table_compare}. Как можно видеть из приведённых асимптотических оценок вычислительной сложности, предложенный метод позволяет достигать эффективной в отношении характера зависимости от размерности задачи $n$ и требуемой точности $\varepsilon$ скорости сходимости, не уступающей другим методам, при некоторой плате за это в виде логарифмического множителя в оценке. 

\begin{table}[ht]
\centering
\begin{tabular}{|c|c|c|c|}
\hline
Метод        & Ит. сложность                 & Выч. сложность       &  Источник \\ \hline
FGM          & $ \mathcal{O}\left(sn\right)$ & $\displaystyle \mathcal{O}\left(sn \cdot \frac{1}{\varepsilon^{1/2}} \cdot \sqrt{L}\right)$ & \cite{nesterov2018lectures} \\ \hline
CDM          & $ \mathcal{O}\left(s\right)$  & $\displaystyle \mathcal{O}\left(sn \cdot \frac{1}{\varepsilon} \cdot \overline{L}\right)$ & \cite{bubeck2014convex} \\ \hline
ACDM         & $ \mathcal{O}\left(n\right)$  & $\displaystyle \mathcal{O}\left(n^2 \cdot \frac{1}{\varepsilon^{1/2}} \cdot \sqrt{\widetilde{L}}\right)$ & \cite{nesterov2017efficiency} \\ \hline
Catalyst CDM & $ \mathcal{O}\left(s\right)$  & $\displaystyle \widetilde{\mathcal{O}}\left(sn \cdot \frac{1}{\varepsilon^{1/2}} \cdot \sqrt{\overline{L}}\right)$ & данная работа \\ \hline
\end{tabular}
\caption{Сравнение эффективности методов}
\label{table_compare}
\end{table}
Заметим, кроме того, что несмотря на существенное сходство оценок, между двумя наиболее эффективными методами в таблице (FGM и Catalyst CDM) существует также различие в константах, характеризующих гладкость функции: $L$~--- в FGM и $\overline{L}$~--- в Catalyst CDM, тем самым поведение рассматриваемого метода для различных задач напрямую зависит от характера их покомпонентной гладкости. В общем случае, нельзя утверждать, что одна из констант асимптотически существенно выгоднее другой, однако в ряде частных случаев возможно выписать оценки значений констант явно, и часто оказывается, что $\overline{L}$ ``меньше`` $L$. Наиболее существенен выигрыш в том случае, если справедливы соотношения $L=\mathcal{O}(n)$, $\overline{L} = \mathcal{O}(1)$~--- тогда в оценке вычислительной сложности предлагаемого метода удаётся редуцировать фактор порядка $\mathcal{O}(\sqrt{n})$ по сравнению с быстрым градиентным методом. В следующем разделе будет рассмотрен пример постановки задачи, в которой данный случай имеет место. 

\section{Численные эксперименты} \label{section2}

В данном разделе описывается характер практического поведения метода, на примере следующей оптимизационной задачи для функционала, имеющего вид SoftMax-функции:
\begin{equation}\label{eq:softmax}
\min\limits_{x \in \mathbb{R}^n}~~ \{f(x) = \gamma \ln \left( \sum_{j=1}^m \exp\left(\frac{\left[A x\right]_j}{\gamma}\right) \right) -\langle b,x \rangle\},
\end{equation}
где $b \in \mathbb{R}^n$, $A \in \mathbb{R}^{m \times n}$. Задачи такого вида существенно важны для многих приложений, в частности, они возникают в задачах энтропийно-линейного программирования в качестве двойственной задачи \cite{chernov2016method, gasnikov2016effective}, в том числе в задаче оптимального транспорта, а также исполняют роль сглаженной аппроксимации функции $\max$ (что и дало функционалу название SoftMax) и, соответственно, нормы $\|\cdot\|_{\infty}$, что может быть востребовано в некоторых постановках задачи PageRank или при решении СЛУ. Притом во всех описанных задачах важным частным случаем являются разреженные постановки, в случае которых матрица $A$ разрежена, то есть такова, что среднее число ненулевых элементов в строке $A_j$ не превосходит некоторого $s \ll n$ (будет удобно также предполагать возможность для одной из строк $A_j$ являться полностью неразреженной).

Сформулируем свойства, которыми обладает функция $f$\; \cite{gasnikov2018modern}:\mynobreakpar\vspace{0.2cm}
\begin{enumerate}
    \item $f$ дифференцируема;\mynobreakpar
    \item $\nabla f$ удовлетворяет условию Липшица с константой $L = \max_{j=1,...,m} \|A_j\|_2^2$;\mynobreakpar
    \item $\nabla_i f$ удовлетворяют покомпонентному условию Липшица с константами $L_i = \max_{j=1,...,m} |A_{j i}|$.
\end{enumerate}

Начнём уточнение свойств с первого пункта. Выпишем выражение для $i$-той компоненты градиента данного функционала:
$$\nabla_i f(x) = \frac{\sum_{j=1}^m{A_{j i} \exp\left(\left[A x\right]_j\right)}}{\sum_{j=1}^m{\exp\left(\left[A x\right]_j\right)}}.$$
Как можно видеть, наивное вычисление этого выражения может занимать время, сравнимое с вычислением градиента в целом, что будет значительно влиять на вычислительную сложность, а значит и на время работы метода. Однако в то же время, многие члены в этом выражении могут перевычисляться либо редко, либо покомпонентно, и использоваться при совершении шага метода как члены его дополнительной последовательности, так что сложность итерации будет оставаться эффективной, и применение покомпонентного метода будет оправдано. Для удобства описания используемых вычислительных приёмов, запишем шаг покомпонентного метода в виде:
$$y_{k+1} = y_{k} + \delta e_i,$$
где $\delta$~--- размер шага, умноженный на соответствующую компоненту градиента, $e_i$~--- $i$-тый орт базиса. 

\begin{enumerate}
    \item Будем хранить набор значений $\left\{\exp\left(\left[A y_k\right]_j\right)\right\}_{j=1}^m$, использующихся для вычисления суммы в числителе. Обновление этих значений после осуществления шага метода имеет сложность $\mathcal{O}(s)$, ввиду того что $A y_{k+1} = A y_k + \delta A_i$, и вместе с тем $A_i$ имеет не более $s$ ненулевых компонент, а значит потребуется вычисление $s$ корректирующих множителей и умножение на них соответствующих значений из набора.
    \item Как можно видеть из первого пункта, производить умножение разреженных векторов следует за $\mathcal{O}(s)$, учитывая лишь ненулевые компоненты. В смысле программной реализации это означает необходимость использования для кэшируемых значений и для строк матрицы $A$ разреженного представления, то есть хранение лишь пар индекс-значение ненулевых элементов~--- тогда, очевидно, сложность арифметических операций для таких векторов будет пропорциональна сложности цикла с элементарными арифметическими операциями, число итераций которого равно числу ненулевых элементов (в языке программирования python, например, такой формат хранения реализуется в методе scipy.sparse.csr\_matrix \cite{scipy}). 
    \item Аналогично, будем хранить значение $\sum_{j=1}^m{\exp\left(\left[A y_k\right]_j\right)}$, являющееся знаменателем. Его обновление осуществляется с той же сложностью, что и обновление последовательности из пункта 1 (путём вычисления суммы ненулевых слагаемых, прибавляемых к каждому значению из набора).
    \item Поскольку вычисление указанного выражения требует вычисления значений экспонент, может происходить переполнение типов. Для решения этой проблемы стандартно применяется exp-normalize trick \cite{blanchard2019accurately}. Однако для его применения следует также хранить значение $\max_{j=1,...,m} \left[A y_k\right]_j$. Вместе с тем, нет необходимости знать именно это значение, или, иначе, знать его точно~--- достаточно лишь его приближения, чтобы значения в показателях экспонент были малы, так что перевычислять эту величину можно гораздо реже: например, раз в $m$ итераций метода, в результате чего амортизированная сложность будет равна $\mathcal{O}(s)$.
\end{enumerate}
Итак, в дальнейших рассуждениях можно полагать, что итерация покомпонентного метода при решении соответствующей вспомогательной задачи, имеет амортизированную сложность $\mathcal{O}(s)$.

Далее, рассмотрим подробнее вопрос о величине констант гладкости данного функционала. Можно выписать асимптотические формулы для $L$ и $\overline{L} = \frac{1}{n} \sum_{i=1}^n L_i$:
$$L = \max_{j=1,...,m} \|A_j\|_2^2 = \mathcal{O}(n), \quad \overline{L} = \frac{1}{n} \sum_{i=1}^n \max_{j=1,...,m} |A_{j i}| = \mathcal{O}(1).$$
Используя эти оценки, уточним вычислительную сложность методов FGM и Catalyst CDM в применении к данной задаче:
$$T_{FGM} = \mathcal{O}\left(sn^{3/2} \cdot \frac{1}{\varepsilon^{1/2}}\right), \quad T_{CCDM} = \widetilde{\mathcal{O}}\left(sn \cdot \frac{1}{\varepsilon^{1/2}}\right).$$
Таким образом, в теории, применение метода Catalyst CDM для решения данной задачи позволяет, по сравнению с FGM, редуцировать в асимптотической оценке вычислительной сложности множитель порядка $\mathcal{O}(\sqrt{n})$. Практически, это означает, что предложенный метод разумно применять для задач большой размерности.

Сравним теперь работу предложеннного в статье метода (Catalyst CDM) с рядом альтернативных подходов: градиентным спуском (GM), быстрым градиентным методом (FGM), покомпонентным спуском (CDM) и ускоренным покомпонентным спуском (ACDM), на примере задачи~\eqref{eq:softmax} с искусственно сгенерированной двумя различными способами матрицей $A$. На рис. \ref{fig:umc1} и \ref{fig:umc2} представлены графики сходимости рассматриваемых методов: по оси абсцисс отложено время работы методов в секундах, по оси ординат~--- невязка по функции в логарифмическом масштабе ($f_*$ находится путём поиска соответствующей точки $x_*$ с помощью метода FGM, настроенного на точность, заведомо значительно превосходящую возможную для достижения на выбранном временном промежутке).

\begin{figure}[!ht]
\begin{minipage}{0.49\textwidth}
    \centering
    \includegraphics[width=\linewidth]{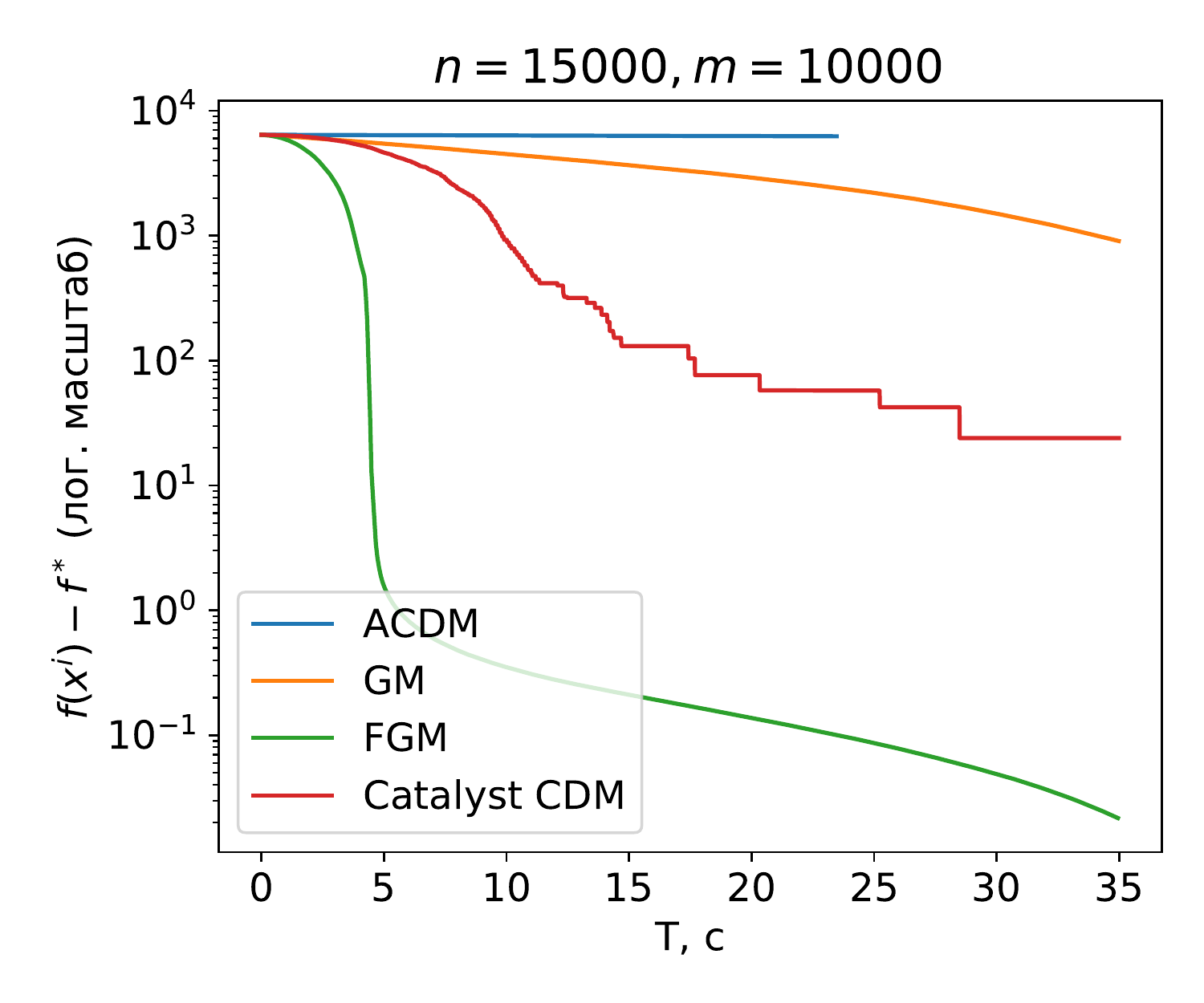}
    \caption{Сходимость методов для задачи SoftMax~\eqref{eq:softmax} с равномерно разреженной случайной матрицей.}
    \label{fig:umc1}
\end{minipage}
\hfill
\begin{minipage}{0.49\textwidth}
    \centering
    \includegraphics[width=\linewidth]{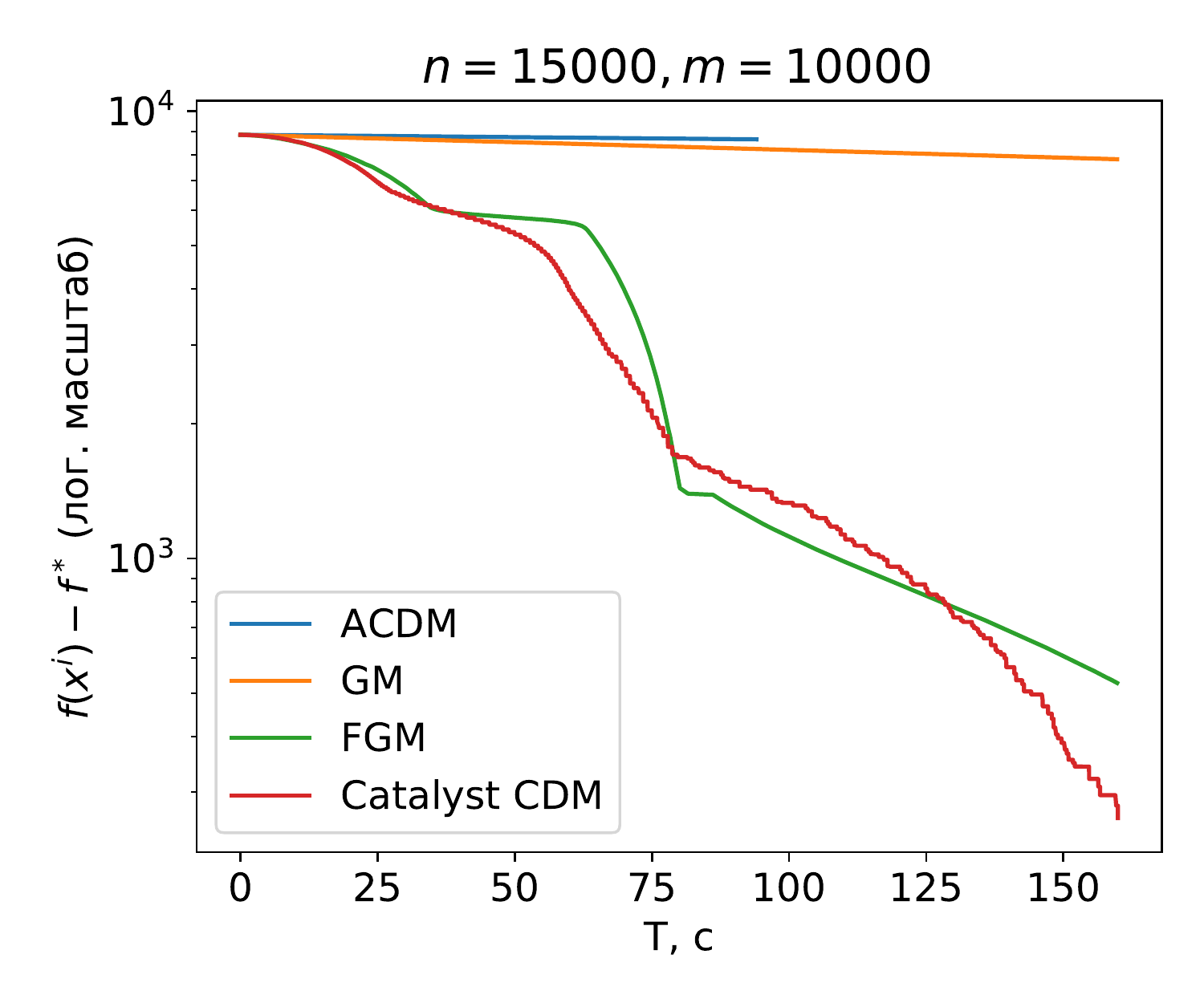}
    \caption{Сходимость методов для задачи SoftMax~\eqref{eq:softmax} с неоднородно разреженной случайной матрицей.}
    \label{fig:umc2}
\end{minipage}
\end{figure}

На рис. \ref{fig:umc1} представлен случай, для которого все элементы матрицы $A$ являются н.о.р. случайными величинами из дискретного равномерного распределения $A_{j i} \in \mathcal{U}\{0, 1\}$, при этом число ненулевых элементов составляет $s \approx 0.2 m$, и параметр $\gamma=0.6$ (так же как и во втором случае). В такой постановке предложенный метод демонстрирует более быструю сходимость по сравнению со всеми сравниваемыми методами, за исключением FGM. В то же время, в постановке, отражённой на рис. \ref{fig:umc2}, при которой число ненулевых элементов, по сравнению с первым случаем, увеличено до $s \approx 0.75 m$, а матрица генерируется неравномерно в соответствии с правилом: $0.9 m$ строк с $0.1 n$ ненулевых элементов и $0.1 m$ строк с $0.9 n$ ненулевых элементов, а также одна и строк матрицы является полностью неразреженной, предложенный метод сходится быстрее FGM. Это объясняется тем, что в этом случае $L = n$, тогда как $\overline{L}$ по прежнему остаётся достаточно мало, в результате чего константа в предложенном методе оказывает заметно меньшее влияние на вычислительную сложность, чем в случае FGM. Из результатов эксперимента также можно отметить, что гораздо существеннее степени разреженности задачи на эффективность предложенного метода влияет характер её покомпонентной гладкости.
\newpage
\section{Заключение} 

В данной работе предлагается вариант покомпонентного метода, ускоренного с помощью универсальной проксимальной оболочки ``ускоренный метаалгоритм``. Проведённый теоретический анализ предложенного метода позволяет утверждать, что зависимость его вычислительной сложности от размерности задачи и требуемой точности решения не уступает прочим методам, используемым для оптимизации выпуклых липшицево гладких функций, а оценка вычислительной сложности сравнима с оценкой быстрого градиентного метода. При этом, в предложенной схеме сохраняются свойства, характерные для классического покомпонентного метода, в том числе возможность использования свойств покомпонентной гладкости функции. Приведённые численные эксперименты подтверждают практическую эффективность метода, и также подчёркивают особенную релевантность предложенного подхода для часто возникающей в различных приложениях задачи оптимизации функции вида SoftMax.

Данная статья представляет результат работы над проектом, предложенным А.В.~Гасниковым в рамках проектной смены\footnote{Ссылка на сайт проектной смены: \url{https://sochisirius.ru/obuchenie/graduates/smena673/3258}} "Современные методы теории информации, оптимизации и управления"\;Сириус 2-23~августа~2020 г. Авторы выражают благодарность организатору проектной смены А.С.~Ненашеву за создание комфортных условий для работы.

\bibliographystyle{plain}
\bibliography{main}

\end{document}